\documentclass[a4paper,12pt]{article}
\usepackage[top=2cm, bottom=2cm, outer=2cm, inner=2cm, headsep=14pt]{geometry}
\usepackage{amsmath,amsfonts}
\usepackage{enumitem}
\usepackage{adjustbox}
\usepackage{blkarray}

\usepackage{tikz}
\usepackage{tkz-graph}
\usepackage{tkz-berge} 
\usetikzlibrary{positioning}

\usepackage{multicol,color} 

\usepackage{tikz} 
\usetikzlibrary{decorations.pathreplacing,decorations.markings} 
\usepackage{url} 
\usepackage{hyperref}

\usepackage{pgfplots}
\usepackage{mathrsfs}
\usetikzlibrary{arrows}

\usepackage{soul,xcolor} 

\tikzset{
  on each segment/.style={
    decorate,
    decoration={
      show path construction,
      moveto code={},
      lineto code={
        \path [#1]
        (\tikzinputsegmentfirst) -- (\tikzinputsegmentlast);
      },
      curveto code={
        \path [#1] (\tikzinputsegmentfirst)
        .. controls
        (\tikzinputsegmentsupporta) and (\tikzinputsegmentsupportb)
        ..
        (\tikzinputsegmentlast);
      },
      closepath code={
        \path [#1]
        (\tikzinputsegmentfirst) -- (\tikzinputsegmentlast);
      },
    },
  },
  mid arrow/.style={postaction={decorate,decoration={
        markings,
        mark=at position .5 with {\arrow[#1]{stealth}}
      }}},
}

\usepackage{avant}

\allowdisplaybreaks

\newcommand{\G}{\Gamma}

\newcommand{\Mat}{\mbox{\rm Mat}}

\def\0{{\boldsymbol 0}}

\def\E{{\cal E}}

\def\J{{\cal J}}

\def\RR{{\mathbb R}}

\def\A{{\mathcal A}}

\def\R{{\mathcal R}}
\def\P{{\mathcal P}}

\def\I{{\mathcal I}}

\DeclareMathOperator{\Span}{span}

\DeclareMathOperator{\rank}{rank}

\DeclareMathOperator{\trace}{trace}

\DeclareMathOperator{\im}{col}

\DeclareMathOperator{\spec}{spec}

\DeclareMathOperator{\OO}{\boldsymbol{O}}

\DeclareMathOperator{\col}{col}

\newtheorem{theorem}{Theorem}[section]
\newtheorem{lemma}[theorem]{Lemma}
\newtheorem{corollary}[theorem]{Corollary}
\newtheorem{proposition}[theorem]{Proposition}
\newtheorem{definition}[theorem]{Definition}

\definecolor{ForestGreen}{RGB}{12, 110, 46}
\definecolor{ForestGreenTwo}{RGB}{120, 110, 86}

\newenvironment{proof}{{\noindent\it Proof. }}{\nopagebreak\hspace*{0.5cm}\hfill$\hbox{\rule{3pt}{6pt}}$\medskip}

\definecolor{ForestGreen}{RGB}{12, 110, 46}
\definecolor{ForestGreenTwo}{RGB}{120, 110, 86}

\newfont\fiverm{cmr5}
\def\eeq{\end{equation}}
\def\lbeq#1{\begin{equation} \label{#1}}

\title{
On the combinatorial structure of graphs with a spectral idempotent of small dual diameter
}

\author{
{Edwin R. van Dam}\\
{\small Department of Econometrics and Operations Research}\\
{\small Tilburg University, the Netherlands}\\
{\small Edwin.vanDam@tilburguniversity.edu} \and
{Giusy Monzillo}\\
{\small Faculty of Mathematics, Natural Sciences}\\
{\small and Information Technologies}\\
{\small University of Primorska}\\
{\small Muzejski trg 2, 6000 Koper, Slovenia }\\
{\small Giusy.Monzillo@famnit.upr.si} \and
{Safet Penji{\'c}}\\
{\small Faculty of Mathematics, Natural Sciences}\\
{\small and Information Technologies; and}\\
{\small Andrej Maru\v{s}i\v{c} Institute}\\
{\small University of Primorska}\\
{\small Muzejski trg 2, 6000 Koper, Slovenia }\\
{\small Safet.Penjic@iam.upr.si}
}

\begin{document}
\maketitle

\begin{abstract}
Let $\G$ be a connected regular graph with an eigenvalue $\lambda$ and corresponding idempotent $E_{\lambda}$. Let $\E_{\lambda}=\langle J,E_{\lambda}\rangle^\circ$ be the algebra generated by $J$ and $E_\lambda$ with respect to the entrywise-Hadamard product, where $J$ is the all-$1$ matrix. We study the combinatorial structure of a graph $\G$ for which $\E_{\lambda}$ has dimension $2$, giving a combinatorial characterization of such graphs in terms of equitable partitions. We present many examples and classify the distance-regular graphs with this property, as well as graphs that generate a $3$-class association scheme. We also study the graphs that have two eigenvalues $\lambda$ for which ${\rm dim}(\E_{\lambda})=2$ and determine all such graphs with four distinct eigenvalues.
\end{abstract}

\begin{center}
{\bf Celebrating the 70th birthday of Paul Terwilliger}
\end{center}

\smallskip
{\small

}


\section{Introduction}

A vector space of symmetric matrices that is closed under both ordinary and entrywise multiplications, containing the identity matrices for both multiplications, i.e., $I$ and $J$, is known to be the Bose-Mesner algebra of an association scheme.

If $\G$ is a relation in a $d$-class association scheme with adjacency algebra $\A=\langle I,A\rangle^\cdot =\langle I,A,A^2,\ldots\rangle$, then $\G$ generates the entire scheme if $\dim \A=d+1$. On the other end, if $\G$ is a regular graph with $\dim \A=2$, then $\G$ must be a disjoint union of cliques. In addition, if $\G$ is a relation in an association scheme, this gives us relatively little information, as we can only conclude that the scheme is imprimitive.

Here we study a dual version of the above. We consider a regular graph $\G$ with an idempotent $E$ (for the ordinary product) in the adjacency algebra $\A$, and the ``idempotent algebra'' $\E = \langle J, E \rangle^\circ=\langle J,E,E\circ E,E^{\circ3},\ldots\rangle$, where $\circ$ denotes entrywise (Hadamard) multiplication. In particular, we consider the case $\dim \E=2$ and study the combinatorial structure that can be obtained from this (note that this is again the other end of dimension $d$, which in the case of a $d$-class scheme gives all combinatorial information).

Before we present our main results, we define what an indubitable partition is.

\begin{definition}
\label{xA}{\rm
A partition $\Pi$ of the vertex set $X$ of a connected graph $\G$ is called {\em indubitable} with parameters $(a,b)$ whether for each cell $\P\in\Pi$ of the partition and each vertex $x$ of $\G$ the number of neighbors of $x$ in $\P$ equals $a$ if $x \in \P$, and $b$ otherwise.}
\end{definition}
If $\Pi$ is an indubitable partition with parameters $(a,b)$ and $r+1$ cells of (the vertex set of) a connected $k$-regular graph $\G$, then its quotient matrix equals $aI+b(J-I)$, which has eigenvalues $k=a+rb$ (with multiplicity $1$) and $\lambda=a-b$ (with multiplicity $r$), which yields that $a=\lambda+\frac{k-\lambda}{r+1}$ and $b=\frac{k-\lambda}{r+1}$. It is well known (see Section~\ref{oA}) that these eigenvalues form a sub(multi)set of the spectrum of $\G$. Of particular interest (and the topic of research of this paper) is the case when the number of cells of the indubitable partition is equal to $m+1$, where $m$ is the multiplicity of the eigenvalue $\lambda$;
we then say that the indubitable partition is {\em full}. In this case, all eigenvectors for $\lambda$ can be obtained from the quotient.

Our main characterization of full indubitable partitions is precisely in terms of the idempotent algebra of the spectral idempotent $E_{\lambda}$ (the projection matrix of the related eigenspace) having dimension $2$.

\begin{theorem}
\label{kD}
Let $\G$ be a connected regular graph on $v$ vertices that has an eigenvalue $\lambda$ with multiplicity $m$ and corresponding spectral idempotent $E_{\lambda}$. Let
$\E=\E_{\lambda}=\langle J,E_{\lambda}\rangle^\circ$. Then $\E$ has dimension $2$ if and only if $\G$ has a full indubitable partition corresponding to $\lambda$.
Moreover, if $\dim \E=2$, then $E_{\lambda}
= \frac{m+1}{v} K - \frac{1}{v} J$,
where
$K$ is permutation-similar to $I_{m+1}\otimes J_{v/(m+1)}$.
\end{theorem}

This paper is further organized as follows. In Section~\ref{2B}, we introduce the notation and further background.
In Section~\ref{fK}, we prove Theorem~\ref{kD} in a series of steps. We also show, in Proposition~\ref{bC}, that if there is an idempotent $E \in \A$ of rank $m$ such that $EJ=0$ and $\dim \langle J,E\rangle^\circ=2$, then $\G$ has an equitable partition into $m+1$ parts.
In Section~\ref{eF}, we present examples of graphs with full indubitable partitions. These include bipartite graphs (Section~\ref{eC}) and some other graphs with a simple eigenvalue as characterized in Proposition~\ref{eL}. We also determine all distance-regular graphs with a full indubitable partition in Proposition~\ref{p_drg} and characterize the examples in a $3$-class association scheme in Proposition~\ref{prop:4evas}. We finish this section by discussing the so-called uniform schemes in Section~\ref{sec:uniform} and a general product construction in Section~\ref{7A}.
In the final Section~\ref{xN}, we consider graphs with two full indubitable partitions. In Theorem~\ref{xQ}, we obtain that the only such graphs with four distinct eigenvalues are the Cartesian products of two complete graphs (that is, grid graphs), which we also characterize combinatorially as certain co-edge-regular graphs. We conclude by determining the bipartite graphs with five distinct eigenvalues and two full indubitable partitions in Theorem~\ref{rC}.

\newpage
\section{Preliminaries}
\label{2B}

\subsection{Graphs and matrices}\label{sec:gm}

In this paper, we consider simple undirected graphs $\G$
with the standard $01$-adjacency matrix $A$.
By the eigenvalues of $\G$ we mean the eigenvalues of $A$. For an eigenvalue $\lambda$, we let $E_{\lambda}$ be the corresponding {\em spectral idempotent}, that is, the matrix representing the projection onto the eigenspace $\ker (\lambda I-A)$.
The {\em adjacency algebra} $\A$ of $\G$ is the vector space $\{f(A) \mid f \mbox{ is a polynomial} \}$ over $\RR$ spanned by the powers of $A$. We note that $\A=\Span\{E_\lambda\mid \lambda \mbox{ is an eigenvalue of }\G\}$.
For more general background and notation on spectra of graphs, we refer to \cite{BH}.

Dual to the adjacency algebra, for an idempotent $E$, we let $\E$ be the {\em idempotent $\circ$-algebra} $\{f(E) \mid f \mbox{ is a polynomial} \}$, where the multiplication $\circ$ is entrywise, in other terms,
$\E = \langle J,E,E\circ E, E\circ E \circ E,\dots \rangle$.

Matrices $A$ and $B$ are called {\em permutation-similar} (denoted by $A \sim B$) if there is a permutation matrix $P$ such that $B=P^{\top}AP$. In other words, $B$ can be obtained from $A$ by simultaneously permuting rows and columns. Two graphs are {\em isomorphic} if their adjacency matrices are permutation-similar. Matrices $A_1$ and $A_2$ are simultaneously permutation-similar to $B_1$ and $B_2$ if there is a permutation matrix $P$ such that $B_1=P^{\top}A_1P$ and $B_2=P^{\top}A_2P$.

For matrices $C = (c_{ij})\in\Mat_{m\times n}(\RR)$ and $B\in\Mat_{p\times q}(\RR)$, the {\em Kronecker product} $C \otimes B\in\Mat_{mp\times nq}(\RR)$ is the (block) matrix obtained by replacing each entry $c_{ij}$ of $C$ with the block $c_{ij} B$.

The {\em Cartesian product} of two graphs with adjacency matrices $A$ and $B$ (on $v$ and $v'$ vertices, respectively) is the graph with adjacency matrix $A \otimes I_{v'} + I_v \otimes B$.
Its eigenvalues are all possible sums of eigenvalues of $A$ and $B$ \cite[\S1.4.6]{BH}.

Let $\G$ be a graph with adjacency matrix $A$. Then the {\em bipartite double} of $\G$ is the graph with adjacency matrix $(J_2-I_2)\otimes A$.

\subsection{Some properties of the quotient matrix of an equitable partition}
\label{oA}

An {\it equitable partition} of a graph $\G$ is a partition $\Pi=\{\P_0, \P_1, \dots, \P_s\}$ of its vertex set, such that for all integers $i,j$ $(0 \le i,j \le s)$ the number of neighbors in $\P_j$ of a vertex $x\in \P_i$ depends only on $i$ and $j$ (and not on $x$). We denote the corresponding parameter by $q_{ij}$, i.e., for every $x\in\P_i$ we have $|\G_1(x)\cap\P_j|=q_{ij}$. The matrix $Q= (q_{ij})_{(s+1)\times(s+1)}$ is called the {\em quotient matrix} of the equitable partition $\Pi$ of $\G$. If $P$ is a matrix having the characteristic vectors of the cells of the partition as columns, then we have that $AP=PQ$. In fact, we have the following characterization.

\begin{lemma}[{\cite[\S9.3]{GR}}]
\label{kB}
Let $P$ be a matrix whose columns are the characteristic vectors of the cells of a partition $\Pi$ of a graph $\Gamma$ with adjacency matrix $A$. Then the following are equivalent.
\begin{enumerate}[label=\rm(\roman*)]
\item $\Pi$ is an equitable partition of $\G$.
\item There is a matrix $Q$ such that $AP=PQ$.
\item $\col P$ is invariant under $A$.
\end{enumerate}
\end{lemma}

It is well known and easy to see that if $x$ is an eigenvector for $Q$ with eigenvalue $\lambda$, then $Px$ is an eigenvector for $A$. This gives rise to eigenvectors that are constant on the cells of the equitable partition. In fact, the other eigenvectors are orthogonal to $P$ (see \cite[\S9.3]{GR}). As mentioned in the introduction, in case the indubitable partition is full corresponding to an eigenvalue $\lambda$, then all eigenvectors of $A$ for $\lambda$ are constant on the cells of the partition.

We note that an indubitable partition as defined in the introduction is equitable, in fact, it is uniformly equitable, i.e., there is a $01$-matrix $\widetilde{A}$ with zero diagonal such that $Q=aI+b\widetilde{A}$  (as defined in {\cite[\S4]{GM} and \cite[\S11.1]{BCN})}).

Finally, we observe that from $AP=PQ$, it follows that $f(A)P=Pf(Q)$ for any polynomial $f$, and hence it follows that an equitable partition for $\G$ is equitable for any matrix in the adjacency algebra $\A$ of $\G$.
(Of particular interest are $01$-matrices within the adjacency algebra.)
Moreover, if $\G$ has an indubitable partition with parameters $(a,b)$, then $Q=aI+b(J-I)$, and it follows that the partition is indubitable for any matrix in $\A$ as well.

\subsection{Association schemes and distance-regular graphs}

A (symmetric) $d$-class {\em association scheme} is a partition of the edge set of the complete graph on $v$ vertices into spanning subgraphs $\G_1,\G_2,\dots, \G_d$ (that are called the non-trivial relations of the scheme) with adjacency matrices $B_1,B_2,\dots, B_d$ such that the vector space $\langle B_0,B_1,B_2,\dots, B_d \rangle$ (with $B_0=I$) is an algebra, that is, it is closed under matrix multiplication.  This algebra is called the {\em Bose-Mesner algebra} of the association scheme, and it is important to note that it is also closed under entrywise multiplication $\circ$ (by construction). It also has a dual basis $\{E_0=\frac1v J,E_1,\dots,E_d\}$ of primitive idempotents. For background and notation on association schemes, we refer to \cite{BIT,BCN}.

\subsubsection{Imprimitivity}
\label{iMp}

An assocation scheme is called imprimitive if at least one of its relations is disconnected. For us, it is relevant (cf. Theorem~\ref{kD}) that an association scheme is imprimitive if and only if there are sets of indices $\I$ and $\J$ (both containing $0$) such that
$$
\sum_{i \in \I} B_i = \tfrac v{m+1} \sum_{j \in \J} E_j \sim I_{m+1}\otimes J_{v/(m+1)}
$$
for some $m$ with $0<m<v-1$ (see, for example, \cite[{\S2}]{DMM} or \cite[Theorem~2.1]{MMW}).
In this case, let $K=\sum_{i \in \I} B_i\sim I_{m+1}\otimes J_{v/(m+1)}$ and let $\Pi$ be the partition of the vertex set whose characteristic vectors are the distinct columns of $K$.
It then follows from Lemma~\ref{kB}(iii) (cf. the proof of Proposition~\ref{bC}) that $\Pi$ is an equitable partition for all relations in the association scheme (since $\im K=\im\sum_{j \in \J} E_j$). In fact, if there is an edge in $\G_h$ between two fixed cells of $\Pi$, then $h \notin \I$ and it is not hard to see that every vertex in one of the two cells has $b_h:=\sum_{i \in \I}p^h_{hi}$ $\G_h$-neighbors in the other cell. This yields that the partition $\Pi$ is uniformly equitable, i.e., its quotient matrix $Q_h$ equals $b_h\widetilde{B_h}$ for some symmetric $01$-matrix $\widetilde{B_h}$ with zero diagonal. The set $\{\widetilde{B_h} \mid h \notin \I \}$ is the set of adjacency matrices of the non-trivial relations of the so-called quotient scheme \cite[\S2.4]{BCN}). Note that some of the $\widetilde{B_h}$ are typically identical and we should therefore not interpret the above set of adjacency matrices as a multiset. Moreover, it follows that for $j \in \J$, the primitive idempotents $E_j$ are constant on the cells of the partition (since each eigenvector for $\lambda_j$ is constant on the cells of the partition, cf. Section~\ref{oA}), and their quotient matrices (appropriately scaled so that they are idempotent) form the primitive idempotents of the quotient scheme (and hence $|\J|-1$ is its number of classes). Note also that the ranks of the primitive idempotents in the quotient scheme are the same as the corresponding ones in the original association scheme.

Further, the above uniformly equitable partition for $\G_h$ is indubitable when $\widetilde{B_h}=J-I$, which is the case precisely if the quotient scheme is trivial (i.e., there is just one relation, the complete graph). In this case it follows that the partition is indubitable for all $h$. (Note that if $h \notin \I$, the parameter $b$ of the indubitable partition is nonzero, while for $h \in \I$, $b = 0$).
Moreover, then $\J$ has size $2$, so there is a primitive idempotent $E_j$ such that $\tfrac v{m+1}(E_0+E_j)$ is permutation-similar to $I_{m+1}\otimes J_{v/(m+1)}$ (cf. Theorem~\ref{kD}.)

For completeness (but quite irrelevant for us), we mention that when we restrict the matrices in $\{B_i \mid i \in \I\}$ to a fixed cell of the partition, then this induces a so-called subscheme on that cell  \cite[{\S2.7.4}]{BIT}).

\subsubsection{Distance-regular graphs}
\label{sec:drgpre}

Assume that $\G$ is a connected graph of diameter $d$.
If $\G_i$ denotes the distance-$i$ graph of $\G$ (that is, vertices are adjacent in $\G_i$ if their distance in $\G$ is $i$), then $\G$ is called {\em distance-regular} if
$\G_1,\G_2,\dots, \G_d$ form an association scheme. In this case, the primitive idempotents of the scheme are the same as the spectral idempotents of $\G$.

A distance-regular graph is called {\em imprimitive} if the corresponding association scheme is imprimitive. A distance-regular graph $\G$ of diameter $d$ is called {\em antipodal} if its distance-$d$ graph $\G_d$ is a disjoint union of cliques.
An imprimitive distance-regular graph with valency at least $3$ is bipartite or antipodal (or both) \cite[Theorem~4.2.1]{BCN}.

Moreover, if $\G$ is bipartite, then the corresponding ``imprimitivity system'' is given by $\I=\{0,2,4,\dots, 2\lfloor \frac d2 \rfloor\}$ and $\J=\{0,d\}$ (with the standard ordering of primitive idempotents). On the other hand, if $\G$ is antipodal, then $\I=\{0,d \}$ and $\J=\{0,2,4,\dots, 2\lfloor \frac d2 \rfloor \}$.
We will use this in Section~\ref{sec:as}.

\section{The indubitable partition}
\label{fK}

In this section, we will prove Theorem~\ref{kD}.
Let $\E=\langle J,E\rangle^\circ$ for some idempotent matrix $E$ and assume that $\dim \E =2$. Since this means that $E$ has just two distinct entries, $\theta_0$ and $\theta_1$ say, we can write $E=\theta_0 K +\theta_1 (J-K)$ for some $01$-matrix $K$. We will now determine the structure of this matrix $K$, which in the end produces the required indubitable partition. We first focus on  idempotent matrices in the adjacency algebra of a regular graph.

\subsection{Idempotents in the adjacency algebra}

\begin{lemma}
\label{1D}
Let $E$ be a symmetric $v \times v$ idempotent matrix of rank $m$ such that $EJ=\OO$. Let $\E=\langle J,E\rangle^\circ$. If $\dim \E=2$, then $E= \frac{m+1}{v} K - \frac{1}{v} J$,
where
$K$ is permutation-similar to $I_{m+1}\otimes J_{v/(m+1)}$.
\end{lemma}

\begin{proof}Let
\begin{equation}
\label{aD}
E=\theta_0 K+\theta_1 (J-K),
\end{equation}
where $K$ is a $01$-matrix and $\theta_0>\theta_1$.

 From $EJ=\OO$ (i.e., $E$ has row sums $0$), it follows that $\theta_0>0>\theta_1$ and that $K$ has constant row sums $-v\theta_1/(\theta_0-\theta_1)$. Furthermore, $E\circ I=\theta_0I$ as $E$ is positive semidefinite, and hence $\theta_0=m/v$ because $\trace E =m$.

It now follows that $E$ is the Gram matrix of $v$ vectors of length $\sqrt{m/v}$ in $\mathbb{R}^m$ with inner products $m/v$ and $\theta_1$. But such vectors have inner product $m/v$ only if they are equal, hence inner product being $m/v$ is an equivalence relation.

Because $K$ has constant row sums and rank $m+1$, it now follows that $K$ is permutation-similar to $I_{m+1}\otimes J_{v/(m+1)}$. Moreover, the row sums imply that
$-v\theta_1/(\theta_0-\theta_1)=v/(m+1)$, which together with $\theta_0=m/v$ yields that $\theta_1=-1/v$.
\end{proof}

Observe that from the above representation of $E$ as a Gram matrix we have $m+1$ distinct vectors of length $\sqrt{m/v}$ in $\mathbb{R}^m$ with mutual angles $\arccos(-1/m)$, that is, a regular simplex. For example, for $m=2$, we have the vertices of an equilateral triangle in the plane (the three vectors have mutual angle $2\pi/3$).

We next consider idempotents $E$ in the adjacency algebra $\A$ of a connected regular graph $\G$ and characterize those for which $\dim \langle J,E\rangle^\circ=2$. We also show that the cliques in the above proof are the cells of an equitable partition.
Note that because the spectral idempotents of $\G$ form a basis of $\A$, an idempotent is in $\A$ if and only if it is a sum of spectral idempotents.

\begin{proposition}
\label{bC}
Let $\G$ be a connected regular graph on $v$ vertices, with adjacency algebra $\A$.
Then $\A$ contains a matrix $K$ that is permutation-similar to $I_{m+1}\otimes J_{v/(m+1)}$ if and only if there is an idempotent $E\in \A$ of rank $m$ such that $EJ=\OO$ and $\dim \langle J,E\rangle^\circ=2$.

Moreover, if so, then the distinct columns of $K$ are the characteristic vectors of an equitable partition of $\G$ into $m+1$ parts of equal size.
\end{proposition}

\begin{proof}
Let $K$ be permutation-similar to $I_{m+1}\otimes J_{v/(m+1)}$ and define
$E=\frac{m+1}v K-\frac1v J$. Then it is easy to show that $E$ is idempotent of rank $m$, with $EJ=\OO$ and $\dim \langle J,E\rangle^\circ=2$.
Moreover, if $K \in \A$, then $E \in \A$. The reverse follows from Lemma~\ref{1D}.

If indeed $K \in \A$, then by the above, $\frac{m+1}v K$ is an idempotent of rank $m+1$, hence it is a sum of spectral idempotents, say $\sum_{j \in \J} E_j$, with $\dim \bigoplus_j \col E_j=\sum_{j \in \J} \rank E_j=m+1=\rank K$. Now $\col K = \col \sum_{j \in \J} E_j \subseteq \bigoplus_j \col E_j$, but because the dimensions are equal, it follows that $\col K =  \bigoplus_j \col E_j$. The latter is clearly invariant under $A$ (in fact, under any matrix in $\A$), so by Lemma~\ref{kB}, the distinct columns of $K$ are the characteristic vectors of an equitable partition.
\end{proof}

A degenerate example of the above is $E=I-\frac1v J$, which is in the adjacency algebra of any connected regular graph of order $v$, and which corresponds to the trivial equitable partition into $v$ parts.

Note that in general, an equitable partition into equal parts does not imply that the corresponding $K$ is in $\A$.
For example, consider the complete bipartite graph $K_{4,4}$ with color classes $\P_0=\{a,b,c,d\}$ and $\P_1=\{a',b',c',d'\}$.
Note that $\spec(K_{4,4})=\{[4]^{1},[0]^{6},[-4]^{1}\}$ and that $\{\P_0,\P_1\}$ is a full indubitable partition. Moreover, $K_{4,4}$ has two equitable partitions into $4$ parts of size $2$ (one is $\{\{a,a'\},\{b,b'\},\{c,c'\},\{d,d'\}\}$, which is indubitable, but not full, and the other is $\{\{a,b\},\{c,d\},\{a',b'\},\{c',d'\}\}$, which is not indubitable).
Both have $K$ permutation-similar to $I_4 \otimes J_2$, which is not in $\langle I,J,(J_2-I_2) \otimes J_4 \rangle$, so $K \notin \A$.

\subsection{Spectral idempotents}

Next, we further specialize and consider spectral idempotents.

\medskip
\begin{proposition}
\label{gDD}
Let $\G$ be a connected $k$-regular graph on $v$ vertices that has an eigenvalue $\lambda$ with multiplicity $m$ and corresponding spectral idempotent $E_{\lambda}$. Let $\E=\langle J,E_{\lambda}\rangle^\circ$.
If $\dim \E=2$, then $\G$ has a full indubitable partition corresponding to $\lambda$.
\end{proposition}

\begin{proof}
By Lemma~\ref{1D}, $E_{\lambda}= \frac{m+1}{v} K - \frac{1}{v} J$,
where
$K$ is permutation-similar to $I_{m+1}\otimes J_{v/(m+1)}$. By Proposition~\ref{bC}, the $m+1$ distinct columns of $K$ are the characteristic vectors of an equitable partition of the vertex set of $\G$. Let $P$ be the $v \times (m+1)$ characteristic matrix of this partition (i.e., its columns are the distinct columns of $K$). Let $Q=\lambda I_{m+1} + \frac{k-\lambda}{m+1}J_{m+1}$.
It then remains to show that $Q$ is the quotient matrix of the partition, i.e., that $AP=PQ$ (see Lemma~\ref{kB}).

Because $K=\frac{v}{m+1}(E_{\lambda}+\frac{1}{v}J)$, we obtain that
$$AK=\tfrac{v}{m+1}(\lambda E_{\lambda} +\tfrac{k}{v}J)= \tfrac{v}{m+1}\big(\lambda (\tfrac{m+1}{v}K-\tfrac{1}{v}J) +\tfrac{k}{v}J\big) = \lambda K + \tfrac{k-\lambda}{m+1}J.$$
Since the columns of $P$ are columns of $K$ and $J_{v\times (m+1)}=PJ_{m+1}$, it thus follows that $AP=\lambda P + \tfrac{k-\lambda}{m+1}J_{v\times(m+1)}=PQ$.
\end{proof}

The final step in proving Theorem~\ref{kD} is to show the reverse of the above proposition by constructing the spectral idempotent $E$ from the full indubitable partition.

\begin{proposition}
\label{pe}
Let $\G$ be a connected $k$-regular graph on $v$ vertices that has an eigenvalue $\lambda$ with multiplicity $m$ and corresponding spectral idempotent $E_{\lambda}$. Let
$\E=\langle J,E_{\lambda}\rangle^\circ$. If $\G$ has a full indubitable partition corresponding to $\lambda$, then $\dim \E=2$.
\end{proposition}

\begin{proof}
Let $P$ be the characteristic matrix of the indubitable partition and $Q=\lambda I_{m+1} + \frac{k-\lambda}{m+1}J_{m+1}$ be the corresponding quotient matrix, so that $AP=PQ$.

Define $K=PP^\top$ (note that $K$ is permutation-similar to $I_{m+1}\otimes J_{v/(m+1)}$). Then $$AK=APP^\top=PQP^\top=P(\lambda I_{m+1} + \tfrac{k-\lambda}{m+1}J_{m+1})P^\top=\lambda K +\tfrac{k-\lambda}{m+1}J.$$
Note that $P^\top P=\frac{v}{m+1}I_{m+1}$, $JP=\frac{v}{m+1}J_{v\times(m+1)}$, $K^2=\frac{v}{m+1}K$, and $KJ=JK=\frac{v}{m+1}J$.

We now define $E=\frac {m+1}v K -\frac 1v J$. Then it is routine to show that $AE=\lambda E$, $E^2=E$, and $\rank E= \trace E =m$.
Thus, $E$ is the spectral idempotent $E_{\lambda}$, so $\dim \E=2$.
\end{proof}

\subsection{Proof of Theorem~\ref{kD}}

Theorem~\ref{kD} follows from Lemma~\ref{1D} and Propositions~\ref{gDD}, \ref{pe}.
The above propositions also imply that if a graph with an eigenvalue $\lambda$ of multiplicity $m$ has a full indubitable partition corresponding to $\lambda$
, then this partition is unique.
In Section~\ref{xN}, we will consider graphs with two different indubitable partitions (corresponding to two distinct eigenvalues).

\section{Some graphs with full indubitable partitions}
\label{eF}

In this section, we give some examples of graphs having a full indubitable partition.
We note that if a graph has a full indubitable partition, then this partition is also full indubitable for its complement (assuming that both the graph and its complement are connected).

\subsection{Bipartite graphs and simple eigenvalues}
\label{eC}

It is well known that a connected $k$-regular graph is bipartite if and only if it has $-k$ as an eigenvalue. If so, then this eigenvalue has multiplicity $1$. Clearly, a bipartite $k$-regular graph has an indubitable partition with two cells (i.e., the color classes) with parameters $(0,k)$. Since $0=\lambda+\frac{k-\lambda}{m+1}$ and $k= \frac{k-\lambda}{m+1}$ are satisfied for $\lambda=-k$ and $m=1$, the indubitable partition is full and hence $\dim \langle J,E_{\lambda}\rangle^\circ =2$.
Moreover, by Theorem~\ref{kD}, the spectral idempotent $E_{\lambda}$ corresponding to eigenvalue $-k$ is permutation-similar to $\frac{1}{v}(2I_2-J_2)\otimes J_{v/2}$, which is a well-known expression for $E_{-k}$.

Also walk-regular graphs with a simple eigenvalue $\lambda$ are known to have an indubitable partition into two parts with parameters $(\lambda+\frac{k-\lambda}{2},\frac{k-\lambda}{2})$ \cite[Theorem~3.3]{ErG}, hence this partition is full. Note that a graph is walk-regular if every matrix in the adjacency algebra has a constant diagonal. Here, we will generalize the above by just assuming that the corresponding idempotent $E_{\lambda}$ has a constant diagonal.

\begin{proposition}
\label{eL}
Let $\G$ be a connected $k$-regular graph with a simple eigenvalue $\lambda\ne k$ and corresponding spectral idempotent $E_{\lambda}$.
Then $\G$ has a full indubitable partition corresponding to $\lambda$ if and only if $E_\lambda$ has constant diagonal.
\end{proposition}

\begin{proof}
Assume first that $\G$ has a full indubitable partition corresponding to $\lambda$. Then, by Theorem~\ref{kD}, $E_{\lambda}$ has constant diagonal.

On the other hand, assume that $E_{\lambda}$ has a constant diagonal. Let $u$ be a normalized eigenvector for $\lambda$. Then $E_{\lambda}=uu^\top$ (as $m=1$), which has diagonal entries $u_i^2$ ($i=1,2,\dots,v$). Thus, $u_i=\pm c$ for some $c$ and all $i$, which implies that $E_\lambda$ has entries $\pm c^2$ only, and hence $\dim \E=2$. By Theorem~\ref{kD}, the result follows.
\end{proof}

\subsection{Association schemes}\label{sec:as}

If $\G$ is a connected graph that is one of the relations in an association scheme and it has a full indubitable partition corresponding to some eigenvalue with multiplicity $m$, then the Bose-Mesner algebra (which contains the adjacency algebra of $\G$) contains a matrix that is permutation-similar to $I_{m+1} \otimes J_{v/(m+1)}$. This implies that the association scheme is imprimitive. Because the partition is indubitable, the corresponding quotient graph of $\G$ is complete, which implies that the quotient scheme is trivial, see Section~\ref{iMp}.

\subsubsection{Distance-regular graphs}

The above implies that if $\G$ is distance-regular with diameter $d$, valency at least $3$, and a full indubitable partition, then it is bipartite or antipodal \cite[Thm.~4.2.1]{BCN}. Moreover, the only matrices in the adjacency algebra that are permutation-similar to $I_{m+1} \otimes J_{v/(m+1)}$ for some $m$ with $0<m<v-1$ are $\sum_{i=0}^{\lfloor d/2 \rfloor}A_{2i}$ (in case $\G$ is bipartite) and $I+A_d$ (in case $\G$ is antipodal) \cite[p.140]{BCN}. We already saw in Section~\ref{eC} that all bipartite graphs have a full indubitable partition. For antipodal distance-regular graphs with diameter $d$, the distance-$d$ graph $\G_d$ is a disjoint union of cliques, but the partition is indubitable only for $d=2$ and $d=3$ (because only then the quotient graph of $\G$ is complete). We thus have the following classification.

\begin{proposition}\label{p_drg}
Let $\G$ be a distance-regular graph of order $v$, with valency $k >2$, diameter $d \ge 2$, with a full indubitable partition corresponding to an eigenvalue $\lambda$ of $\G$ with multiplicity $m$. Then $\G$ is bipartite and $\lambda=-k$, or $\G$ is a complete $(m+1)$-partite graph and $\lambda=-v/(m+1)$, or $\G$ is a distance-regular antipodal cover of a complete graph $K_{m+1}$ and $\lambda=-1$ (with $d=3$).
\end{proposition}

For completeness, we also mention that the complete graph has a (trivial) full indubitable partition into $v=m+1$ parts.

The only cycles (distance-regular graphs with valency $2$) that can have a full indubitable partition are the even cycles (which are bipartite) and the cycles $C_{3n}$, with $n \in \mathbb{N}$, because all eigenvalues have multiplicity at most $2$. The latter have a full indubitable partition into $3$ parts with parameters $(0,1)$, corresponding to eigenvalue $-1$. It also follows that $C_{6n}$, with $n \in \mathbb{N}$ therefore has two full indubitable partitions; more on this in Section~\ref{xN}.

\subsubsection{Three-class association schemes}
Let us next consider $3$-class association schemes. In the generic case (i.e., it is not amorphic, nor is one of the relations a complete multipartite graph), it has as one of its relations a connected graph $\G$ with four distinct eigenvalues. In this case, $\G$ generates the association scheme in the sense that its adjacency algebra is the Bose-Mesner algebra of the scheme (because their dimensions are the same), and hence the spectral idempotents of $\G$ are the primitive idempotents of the scheme.

If the association scheme is imprimitive, then one of the other two relations is a disjoint union of, $m+1$ say, cliques (i.e., $|\I|=2$, for otherwise $\G$ would be complete multipartite). Now the corresponding quotient scheme is trivial\footnote{The quotient scheme is trivial unless it is a $2$-class scheme, in which case the quotient graphs in $\{\widetilde{B_h} \mid h \notin \I \}$ form a pair of complementary strongly regular graphs, so the original $3$-class scheme is the wreath product of a $2$-class scheme and a $1$-class scheme, and $\G$ is the so-called coclique-extension of a strongly regular graph.} (i.e., $|\J|=2$) if and only if the corresponding partition is indubitable for $\G$ with parameters $(0,k/m)$, where $k$ is the valency of $\G$. Moreover, in this case one of the primitive idempotents of the scheme is of the required form (cf. Theorem~\ref{kD}), see Section~\ref{iMp}, so the indubitable partition is full, corresponding to eigenvalue $-k/m$.

On the other hand, let $\G$ be a connected $k$-regular graph with four distinct eigenvalues that has an indubitable partition with parameters $(0,k/m)$ (a so-called regular coloring). Then it was shown in \cite[Thm.~5.7]{Etc} that $\G$ generates a $3$-class association scheme if $-k/m$ is an eigenvalue of $\G$ with multiplicity $m$. We may thus conclude the following.

\begin{proposition}\label{prop:4evas}
Let $\G$ be a connected $k$-regular graph with four distinct eigenvalues and an indubitable partition $\Pi$ with parameters $(0,b)$. Then $\G$ generates an imprimitive $3$-class association scheme in which one of the other relations is a disjoint union of $k/b+1$ cliques if and only if the indubitable partition $\Pi$ is full.
\end{proposition}

\subsubsection{Uniform association schemes} \label{sec:uniform}
In his paper on $4$-class imprimitive association schemes, Higman \cite{H4} defined uniform association schemes.
Informally, an association scheme is called uniform if it is imprimitive with a partition of the vertices into cells such that all intersection numbers are spread uniformly over the cells. In particular, this means that the partition is indubitable for all (connected) relations of the scheme. All bipartite schemes and all cometric Q-antipodal schemes are uniform, for example. Van Dam, Martin, and Muzychuk \cite{DMM} showed that in a uniform association scheme on $v$ vertices, there is a primitive idempotent $E$ such that $\frac1v J+E$ is permutation-similar to $I_{m+1} \otimes J_{v/(m+1)}$ for some $m$ (among other results about the other primitive idempotents), and hence the indubitable partition is full for all relations that generate the scheme. Prominent examples are obtained from so-called linked systems of symmetric designs and linked systems of strongly regular designs (such as a linked system of Van Lint-Schrijver partial geometries). For many more examples, we refer to \cite{DMM}.

\subsection{The Cartesian product of a regular graph and a complete graph}
\label{7A}

The Cartesian product of an $\ell$-regular graph and a complete graph has an indubitable partition with parameters $(\ell,1)$. Under some conditions, this partition is full.

\begin{proposition}
\label{cB}
Let $C$ be the adjacency matrix of a connected $\ell$-regular graph on $m$ vertices and let $\lambda=\ell-1$, with $\ell >1$.
Let $\G$ be the graph with adjacency matrix
$$
A=I_{m+1}\otimes C + (J_{m+1}-I_{m+1})\otimes I_m,
$$
i.e., it is the Cartesian product of $\G(C)$ and a complete graph on $m+1$ vertices.
If $m+1>2\ell$, then $\lambda$ is an eigenvalue of $\G$ with multiplicity $m$ and $\G$ has a full indubitable partition corresponding to $\lambda$.
\end{proposition}

\begin{proof}
By construction, $\G$ has an indubitable partition into $m+1$ parts with parameters $(\ell,1)$.
The eigenvalues of the Cartesian product $\G$ are all possible sums of an eigenvalue of $C$ and an eigenvalue of the complete graph $K_{m+1}$ \cite[\S1.4.6]{BH}. If $\theta$ is an eigenvalue of $C$, then the corresponding eigenvalues $\theta+m$ and $\theta-1$ of $\G$ are equal to $\lambda$ only if $\theta=\ell-1-m$ or $\theta=\ell$. Since $K_{m+1}$ has eigenvalue $-1$ with multiplicity $m$, the second case indeed occurs with multiplicity $m$.
Because $\theta \ge -\ell$ and $m+1>2\ell$ (by assumption), the first case cannot occur, though.
It thus follows that $\G$ has an eigenvalue $\lambda=\ell-1$ of multiplicity $m$, which shows that the indubitable partition corresponding to $\lambda$ is full.
\end{proof}

\begin{corollary}
\label{eB}
Let $m\ge 4$ and $\lambda=1$. Let $\G$ be the Cartesian product of a cycle of length $m$ and a complete graph on $m+1$ vertices.
Then $\lambda$ is an eigenvalue of $\G$ with multiplicity $m$ and $\G$ has a full indubitable partition corresponding to $\lambda$.
\end{corollary}

\begin{proof}
Immediate from Proposition~\ref{cB}.
\end{proof}

\section{Two indubitable partitions}
\label{xN}

In this section, we study the combinatorial structure of graphs that have indubitable partitions corresponding to two distinct eigenvalues.

\begin{proposition}
\label{xO}
Let $\G$ be a connected $k$-regular graph on $v$ vertices.
Assume that $\G$ has
full indubitable partitions  corresponding to eigenvalues $\lambda$ and $\lambda'$ of $\G$ with multiplicities $m$ and $m'$, respectively.
Let $E_{\lambda}$ and $E_{\lambda'}$ be the corresponding spectral idempotents for the eigenvalues $\lambda$ and $\lambda'$, and let $K$ and $K'$ be such that
$E_{\lambda}= \frac{m+1}{v} K - \frac{1}{v} J$ and $E_{\lambda'}= \frac{m'+1}{v} K' - \frac{1}{v} J$. Then $(m+1)(m'+1) \mid v$, and $K$ and $K'$ are simultaneously permutation-similar to
$$
J_{m'+1}\otimes I_{m+1}\otimes J_{v/((m+1)(m'+1))}
\qquad \text{and}\qquad
I_{m'+1}\otimes J_{m+1}\otimes J_{v/((m+1)(m'+1))}.
$$
\end{proposition}

\begin{proof}
Let $\{\P_i\}_{i=0}^{m}$ and
$\{\R_j\}_{j=0}^{m'}$ be the indubitable partitions corresponding to the eigenvalues $\lambda$ and $\lambda'$, respectively, and let $\ell=\frac{v}{(m+1)(m'+1)}$. Recall that $KJ=\frac{v}{m+1}J$ and $JK'=\frac{v}{m'+1}J$. Then we have that
$$v^2E_{\lambda}E_{\lambda'}=((m+1) K - J)
((m'+1) K' -  J)
=(m+1)(m'+1) KK'-vJ.$$
Because $E_{\lambda}E_{\lambda'}=\OO$, it follows that
$KK'=\ell J$.

Pick $i,j$ $(0\le i\le m,~0\le j\le m')$ and arbitrary $x\in\P_i$ and $y\in\R_j$. Note that, $|\P_i\cap\R_j|=(K K')_{xy}=\ell$, hence it is clear that $(m+1)(m'+1)\mid v$.
Next, we rename the vertices in such a way that the $\ell$ vertices in $\P_i \cap \R_j$ are $\{ j\ell(m+1) + i\ell + h \}_{h=1}^{\ell}$. Then $\R_j = \{ j\ell(m+1)+k \}_{k=1}^{\ell(m+1)}$, which yields that (after this reordering) $K' = I_{m'+1} \otimes J_{\ell(m+1)} = I_{m'+1} \otimes J_{m+1} \otimes J_{\ell}$. Moreover, $$
\P_i=\{j\ell(m+1)+i\ell+1,j\ell(m+1)+i\ell+2,\ldots,j\ell(m+1)+i\ell+\ell \mid 0\le j\le m'\}$$ and the result follows.
\end{proof}

\subsection{Four eigenvalues}\label{sec:4ev}

Terwilliger \cite{Ter} showed that the local graphs of thin $Q$-polynomial distance-regular graphs are co-edge-regular with at most five distinct eigenvalues. Note that all imprimitive $Q$-polynomial distance-regular graphs are thin (see \cite[\S5.6]{DKT}). As it turns out, the Cartesian products of two complete graphs (also called grid graphs) are co-edge-regular with 3 or 4 distinct eigenvalues, and they are indeed the local graph of some thin $Q$-polynomial distance-regular graphs, namely the Johnson graphs. Gebremichel, Cao, and Koolen \cite{GCK} obtained several combinatorial characterizations of grid graphs.

It is also clear that grid graphs have two indubitable partitions.
Note that by Proposition~\ref{p_drg}, a strongly regular graph cannot have two full indubitable partitions. However, if the number of eigenvalues is four, then the indubitable partitions are full, as we shall see next.

\begin{theorem}
\label{xQ}
Let $\G$ be a connected $k$-regular graph on $v$ vertices with four distinct eigenvalues. Let $\lambda$ and $\lambda'$ be two of these eigenvalues, with multiplicities $m$ and $m'$, and corresponding spectral idempotents $E_{\lambda}$ and $E_{\lambda'}$, respectively. Let $\E_{\lambda}=\langle J,E_{\lambda}\rangle^\circ$ and $\E_{\lambda'}=\langle J,E_{\lambda'}\rangle^\circ$. Then the statements {\rm(i)--(iv)} are equivalent.

\begin{enumerate}[label=\rm(\roman*)]
\item $\dim \E_{\lambda}=\dim \E_{\lambda'}=2$,
\item $\G$ has full indubitable partitions corresponding to $\lambda$ and $\lambda'$,
\item  $\G$ or its complement is a co-edge-regular $(v,k,\mu)$-graph with $v=(m+1)(m'+1)$, $k=m+m'$, $\mu=2$, such that any two adjacent vertices have $m-1$ or $m'-1$ common neighbors,
\item $\G$ or its complement is the Cartesian product of $K_{m+1}$ and $K_{m'+1}$.
\end{enumerate}
Moreover, if the above statements hold, then $m \neq m'$ and $\G$ generates a $3$-class association scheme with Bose-Mesner algebra $\langle I \rangle + \E_{\lambda}+\E_{\lambda'}$.
\end{theorem}

\begin{proof}
We note first that if $\G$ is the the Cartesian product of $K_{m+1}$ and $K_{m'+1}$, then it has eigenvalues $m'-1$ and $m-1$ with multiplicities $m$ and $m'$, respectively (the other two eigenvalues are $k=m+m'$ and $-2$), and it follows that $\G$ has two corresponding full indubitable partitions.
Indeed, one indubitable partition has $m+1$ parts and parameters $(m',1)$, which is full corresponding to eigenvalue $\lambda=m'-1$ because it has multiplicity $m$. Similarly, this holds for the other partition.
It thus follows easily that (iv) implies (i)--(iii).

Next, we will show that (i) --- or equivalently (ii) --- implies (iv).
As before, we let $K$ and $K'$ be such that
$E_\lambda = \frac{m+1}{v}K - \frac{1}{v}J$ and
$E_{\lambda'} = \frac{m'+1}{v}K' - \frac{1}{v}J$.

Consider the matrices $K-I$, $K'-I$, and $J-I-(K-I)-(K'-I)$. Each has zero diagonal, and it follows from Proposition~\ref{xO} that for each of the three matrices, there is a $1$ entry in a position where the other two matrices have a $0$ entry. Because the adjacency matrix $A$ of $\G$ has four distinct eigenvalues, this implies that $\A=\langle I,K-I,K'-I, J-I-(K-I)-(K'-I) \rangle$ and
$$A=\alpha_1 (K-I) + \alpha_2 (K'-I) + \alpha_3 (J+I-K-K')
$$ for some $01$-coefficients $\alpha_i$ $(1\le i\le 3)$.

If $\alpha_3=\alpha_1$ (in which case $A \in \langle J,I,K' \rangle$) or similarly $\alpha_3=\alpha_2$ (in which case $A \in \langle J,I,K \rangle$), then $A$ has at most three distinct eigenvalues, which is a contradiction. Thus, two cases remain ($\alpha_1=\alpha_2 \ne \alpha_3$), namely $A=K-I+K'-I$ or $A=J+I-K-K'$.
In order for these to be $01$-matrices, we must have $v=(m+1)(m'+1)$ (that is, $\ell=1$ in Proposition~\ref{xO}), in which case $\G$ indeed is the Cartesian product of the complete graphs $K_{m+1}$ and $K_{m'+1}$ or its complement. Moreover, this shows that $m \ne m'$, for otherwise $\G$ would have only three distinct eigenvalues.

Next, we consider $\A=\langle I \rangle + \E_{\lambda}+\E_{\lambda'}=\langle I,J,E_{\lambda},E_{\lambda'} \rangle$. Clearly, this algebra contains $I$ and $J$ and is closed under ordinary multiplication. By using Proposition~\ref{xO}, it follows that
\begin{align*}
v^2E_\lambda \circ E_{\lambda'} &=((m+1) K - J)\circ
((m'+1) K' -  J)\\
& =(m+1)(m'+1)I-(m+1)K-(m'+1)K'+J\\
&= vI-J-vE_\lambda-vE_{\lambda'}.
\end{align*}
Because $\E_{\lambda}$ and $\E_{\lambda'}$ are closed under entrywise product by definition, this shows that $\A$ is also closed under the entrywise product, and hence $\A$ is the Bose-Mesner algebra of an association scheme that is clearly generated by $\G$.

What remains is to show that (iii) implies (iv). Let $\G$ therefore be co-edge-regular as in (iii), with adjacency matrix $A$.
Then every vertex has $mm'$ vertices at distance $2$. By standard counting of edges in the first subconstituent (i.e., the neighborhood) of a fixed vertex $x$ (the numbers of edges between subconstituents and in the second subconstituent are all known), it follows that each vertex $x$ has $m$ neighbors $y$ such that $x$ and $y$ have $m-1$ common neighbors, and $m'$ neighbors $z$ such that $x$ and $z$ have $m'-1$ common neighbors.

Define $B_1$ as the $01$-matrix such that $(B_1)_{xy}=1$ if $x$ and $y$ are adjacent with $m-1$ common neighbors and let $B_0=I$, $B_2=A-B_1$, and $B_3=J-I-A$, then it follows from \cite[Theorem~6.3]{FMPS} that $\{B_0,B_1,B_2,B_3\}$ is the standard basis of an association scheme with Bose-Mesner algebra $\A$. (Alternatively, note that \cite[Theorem~5.1]{Etc} applies to $B_3$.)

Using the standard notation for the valencies $k_i$ and intersection numbers $p^h_{ij}$ of an association scheme, first note that $k_1=m$ and $k_2=m'$ (see above). Assume without loss of generality that $m>m'$ (note that $m$ and $m'$ must be distinct, for otherwise the graph would be strongly regular and have only $3$ eigenvalues).

Recall that $k_h p^{h}_{ij}=k_jp^j_{ih}$, which implies that $k_2p^2_{23}=k_3p^3_{22}$. Suppose now that $p^3_{22}>0$. Then $p^2_{23} \ge k_3/k_2=m>m'=k_2$, which is a contradiction. Thus, $p^3_{22}=0$, which yields $p^2_{23}=0$. Since $p^2_{13}+p^2_{23}=k-1-(m'-1)=m$, it follows that $p^2_{13}=m$.
Now $p^2_{11}+p^2_{12}=k_1-p^2_{13}=0$, so $p^2_{11}=p^2_{12}=0$ and hence also $p^1_{12}=p^1_{22}=0$. Since $p^1_{11}+p^1_{12}+p^1_{21}+p^1_{22}=m-1$, it follows that $p^1_{11}=m-1$ and similarly $p^2_{22}=m'-1$. This implies that $B_1$ is a disjoint union of $(m+1)$-cliques and $B_2$ is a disjoint union of $(m'+1)$-cliques, which shows (iv).
\end{proof}

We note that the generated association scheme in the above theorem is the well-known \emph{rectangular scheme} $R(m+1,m'+1)$. It was first introduced by Vartak \cite{Va}.

\begin{corollary}
\label{rE}
Let $\G$ be a connected $k$-regular bipartite graph on $v$ vertices with four distinct eigenvalues. If $\G$ has a full indubitable partition corresponding to an eigenvalue $\lambda \ne -k$, then $\G$ is distance-regular with intersection array
$(m,m-1,1; 1, m-1, m)$,
i.e., it is the bipartite double of the complete graph $K_{m+1}$.
\end{corollary}

\begin{proof} It follows from Section~\ref{eC} that $\G$ has a second full indubitable partition for the eigenvalue $-k$ with multiplicity $m'=1$; so, by Theorem~\ref{xQ}, $\G$ is the complement of the Cartesian product of $K_{m+1}$ and $K_2$ (because it is bipartite), which is distance-regular as stated.
\end{proof}

Note that for this bipartite distance-regular graph, we have that $v=2(m+1)$. Indeed, it has another full indubitable partition, into $m+1$ parts with parameters $(0,1)$, for the eigenvalue $\lambda=-1$.

\subsection{Bipartite with five eigenvalues}
\label{xZ}

Besides graphs with four eigenvalues, we can also characterize bipartite graphs with five distinct eigenvalues that have an additional full indubitable partition. Indeed, we will show that in such a case $\G$ must be the bipartite double of a complete multipartite graph. Before doing so, we first introduce some notation and make some general observations.

In particular, we will first show that no bipartite graph can have a full indubitable partition corresponding to eigenvalue $0$.

Let $\G$ be a connected regular bipartite graph on $v$ vertices. Let $N$ be the incidence matrix of $\G$, that is, we can write its adjacency matrix $A$ (after a suitable ordering of the vertices) as
$$A=\begin{bmatrix}O& N\\ N^{\top}& O \end{bmatrix}.$$
Using the corresponding bipartition of $\G$ and $L_2:=J_2-I_2$, we also introduce the matrices
$$F=\begin{bmatrix}J_{v/2}& -J_{v/2}\\ -J_{v/2}& J_{v/2} \end{bmatrix} \text{~and~} L=J-F=\begin{bmatrix}O& J_{v/2}\\ J_{v/2}& O \end{bmatrix}=L_2 \otimes J_{v/2}.$$
It is well-known that eigenvectors of $\lambda$ can easily be transformed into eigenvectors of $-\lambda$ and vice versa using the bipartition of $\G$. It follows that $E_{-\lambda} = F \circ E_{\lambda}$ for each eigenvalue $\lambda$.

\begin{proposition}\label{prop:nozero} Let $\G$ be a connected bipartite graph with an eigenvalue $0$. Then $\G$ does not have a full indubitable partition corresponding to $0$.
\end{proposition}

\begin{proof} Take $\lambda=0$. From the above, we find that $E_{\lambda} = F \circ E_{\lambda}$, hence $L \circ E_{\lambda}=0$, which implies that $E_{\lambda}$ has an entry zero. Thus, $E_{\lambda}$ must have at least three distinct entries, which proves the result.
\end{proof}

Let $\G$ now be a connected $k$-regular bipartite graph on $v$ vertices with five distinct eigenvalues. Because bipartite graphs have a symmetric spectrum, the distinct eigenvalues are $\pm k,\pm \lambda$, and $0$, with $\pm \lambda$ having multiplicity $m$, say. The spectral decomposition of $A$ now gives that
$$A=kE_k+\lambda E_\lambda -\lambda E_{-\lambda} -kE_{-k}
=L \circ (\frac{2k}{v} J_v+2\lambda E_\lambda).$$

\begin{theorem}
\label{rC}
Let $\G$ be a connected $k$-regular bipartite graph on $v$ vertices with five distinct eigenvalues. If $\G$ has a full indubitable partition corresponding to an eigenvalue $\lambda \ne -k$ with multiplicity $m$, then $\lambda=-k/m$ and $\G$ is the bipartite double of a complete $(m+1)$-partite graph.
\end{theorem}

\begin{proof}
First we note that $\lambda \ne 0$ by Proposition \ref{prop:nozero}. Again, it follows from Section~\ref{eC} that $\G$ has a second full indubitable partition for the eigenvalue $-k$ with multiplicity $m'=1$, and so Proposition~\ref{xO} yields that
$E_{\lambda}= \frac{m+1}{v} K - \frac{1}{v} J_v$ and $E_{-k}= \frac{2}{v} K' - \frac{1}{v} J_v$, where we reorder the vertices such that
$$
K=J_{2}\otimes I_{m+1}\otimes J_{\ell},\qquad \text{and} \qquad
K'= I_{2}\otimes J_{m+1}\otimes J_{\ell}
$$ (with $\ell=v/(2(m+1))$).
Note that we were careful to have the same ordering of vertices (that is, according to the bipartition) as above, so that $E_{-k}= \frac{2}{v} K' - \frac{1}{v} J_v= \frac1v F$. Thus, from the spectral decomposition, we obtain that
\begin{align*}
vA&=L \circ (2k J_v+ 2v\lambda E_\lambda)
= L \circ (2(k-\lambda)J_v +2\lambda(m+1)K) \\
&= L\circ(
2(k-\lambda)J_v +
2\lambda(m+1)J_2\otimes I_{m+1}\otimes J_{\ell})\\
&= (L_2\otimes J_{v/2}) \circ (J_2 \otimes
[2(k-\lambda)J_{m+1} +2\lambda(m+1) I_{m+1}]\otimes J_{\ell})\\
&= L_2 \otimes [2(k-\lambda)J_{m+1} +2\lambda(m+1) I_{m+1}]\otimes J_{\ell}.
\end{align*}
(The line above is due to the mixed-product property $(A\otimes B) \circ (C \otimes D)=(A \circ C)\otimes(B \circ D)$, for matrices $A, B, C$, and $D$ such that $A$ has same size as $C$ and $B$ has same size as $D$).
Because $A$ is a $01$-matrix, and $\lambda \notin \{0,-k\}$, it follows that $2(k-\lambda)=v$ and $2\lambda(m+1)=-v$, which yields $\lambda=-k/m$.

Now $A$ can be further simplified as $A=L_2 \otimes (J_{m+1}-I_{m+1}) \otimes J_{\ell}$, i.e., $\G$ is the bipartite double of a complete $(m+1)$-partite graph.
\end{proof}

Note that, similarly to Theorem~\ref{xQ}, it can easily be shown that $\langle I,J,F,E_{\lambda},E_{-\lambda} \rangle$ is closed under entrywise multiplication, and hence it is the Bose-Mesner algebra of a $4$-class association scheme. Needless to say that this scheme is generated by $\G$ (but the other relations are not used in the above proof).

Note also that the case $\ell=1$ is degenerate and gives the graph of Corollary \ref{rE}.\\

\section*{Acknowledgments}

This work is supported in part by the Slovenian Research Agency (research program P1-0285 and research projects J1-3001 and N1-0353).

Edwin van Dam thanks \v{S}tefko Miklavi\v{c}, Giusy Monzillo, Safet Penji\'{c}, and the University of Primorska for the kind hospitality during visits in the summers of 2024 and 2025.

\bigskip
{\bf Data availability statement.} No datasets were generated or analyzed during the current study.



\bigskip
{\small

}

\end{document}